\documentclass[12pt,a4paper,reqno,dvipdfmx]{amsart}
\usepackage[top=30truemm,bottom=30truemm,left=30truemm,right=30truemm]{geometry} 
\usepackage{amssymb}
\usepackage{amsthm}
\usepackage{amsmath}
\usepackage{color}
\usepackage[driverfallback=dvipdfm]{hyperref}
\usepackage{amsrefs}
\usepackage{cleveref}
\usepackage{tikz}
\usepackage{cases}
\usepackage{mathtools}

\numberwithin{equation}{section}

\theoremstyle{definition}
 \newtheorem{theorem}{Theorem}[section]
 \crefname{theorem}{Theorem}{Theorems}
 \newtheorem{proposition}[theorem]{Proposition}
 \crefname{proposition}{Proposition}{Propositions}
 \newtheorem{lemma}[theorem]{Lemma}
 \crefname{lemma}{Lemma}{Lemmas}
 \newtheorem{corollary}[theorem]{Corollary}
 \crefname{corollary}{Corollary}{Corollaries}
 \newtheorem{conjecture}[theorem]{Conjecture}
 \crefname{conjecture}{Conjecture}{Conjectures}
 \newtheorem{question}[theorem]{Question}
 \crefname{question}{Question}{Questions}
 
 \crefname{problem}{Problem}{Problems}
 \newtheorem{remark}[theorem]{Remark}
 \crefname{remark}{Remark}{Remarks}

\theoremstyle{definition} 
 \newtheorem{definition}[theorem]{Definition}
 \crefname{definition}{Definition}{Definitions}
 
 \crefname{example}{Example}{Examples}
 
 \crefname{caution}{Caution}{Cautions}
 
 \crefname{equation}{formula}{formulas}

\newcommand{\Tr}{\text{Tr}}

\newcommand{\abs}[1]{\left\lvert#1\right\rvert}

\newcommand{\A}{\mathbb{A}}

\newcommand{\C}{\mathbb{C}}

\newcommand{\Q}{\mathbb{Q}}

\newcommand{\Z}{\mathbb{Z}}

\newcommand{\sgn}{\textup{sgn}}

\DeclareMathOperator{\Gal}{Gal} 

\title{Some periodic integer continued fraction expansions of $\sqrt{m}$ and application to the Pell equations}
\author{Yoshinori Kanamura}
\address{Y. Kanamura Department of Mathematics \\ Faculty of Science and Technology, Keio University, 3-14-1, Hiyoshi, Kohoku, Yokohama, Kanagawa, Japan}
\address{Mathematical Science Team \\ RIKEN Center for Advanced Intelligence Project (AIP)\\ 1-4-1 Nihonbashi, Chuo-ku, Tokyo 103-0027, Japan.}
\email{kana1118yoshi@keio.jp}
     
\author{Hyuga Yoshizaki} 
\address{H. Yoshizaki \\ Department of Mathematics \\ Graduate School of Science and Technology \\ Tokyo University of Science, 2641 Yamazaki, Noda, Chiba, Japan}
\email{yoshizaki.hyuga@gmail.com}


\begin{document}

\begin{abstract}
Periodic integer continued fractions (PICFs) are generalization of the regular periodic continued fractions (RPCFs).
It is classical that a RPCF expansion of an irrational number is unique.
However, it is no longer unique for a PICF expansion. 
Hence it is a natural problem to determine all PICF expansions of irrational numbers.  
In this paper, we determine certain types PICF expansions of square roots of positive square-free integers.
To obtain this result, it plays an important role to determine integer points on certain PCF varieties appeared in \cite{BEJ}.
As an application of these results,
we obtain fundamental solutions of the Pell equations from PICF expansions of square roots of positive square-free integers as well as the RPCF expansions. 
\end{abstract}

\maketitle

\section{Introduction}
For an integer sequence $\{a_n\}_{n\ge0}$ ($a_n\ge1$ for $n\ge1$),
\begin{align}\label{contifrac}
[a_0,a_1,a_2,\ldots]=a_0+\cfrac{1}{a_1+\cfrac{1}{a_2+\cdots}}  
\end{align}
denotes a regular continued fraction (RCF for short).
A RCF (\ref{contifrac}) is periodic if $a_k=a_{l+k}$ for some $l \in \Z_{\geq 1}$, $N \in \Z_{\geq 0}$ and all $k \geq N$.
In the following, we assume that $l$ is the period, that is, $l$ is the smallest integer satisfying the periodic condition. 
Let 
\begin{align*}
    [a_0,\dots,a_{N-1},\overline{a_{N},\dots,a_{N+l-1}}]:=[a_0,\dots,a_{N-1},a_{N},\dots,a_{N+l-1},a_{N},\dots,a_{N+l-1},\dots]  
\end{align*}
denote a regular periodic continued fraction (RPCF for short) and call it a $(N,l)$-type RPCF.
It is well-known that every irrational number has a unique RPCF expansion, 
and every quadratic irrational number has a $(N,l)$-type RPCF expansion for some $N,l$. 

In this paper, we consider periodic integer continued fractions (PICFs for short).
Here, an integer continued fraction is (\ref{contifrac}) in which all $a_n$ are integers (not necessarily positive).
Integer continued fractions (including PICF) have appeared in various researches in geometry (e.g.\ \cite{Beardon-Hockman-Short,BEJ,Kauffamn-Lambropoulou,Short-Stanier}), number theory (e.g.\ \cite{Tong1992,Tong1994,Williams-Buhr}) and dynamical systems (e.g.\ \cite{Tanaka-Ito,Williams1979,Jager1982,GMR}). 
Of course, the unique RPCF expansion of an irrational number $x$ is a PICF expansion. 
However, there are other PICF expansions of $x$ in general. 
For example, we obtain 
\begin{align*}
\sqrt{2} = [1, \overline{2}] = [-1,\overline{1,-2,1}]. 
\end{align*}
Hence it is natural to consider the following question. 

\begin{question}\label{question1}
For each non-negative integer $N$ and positive integer $l$, can we determine all $(N, l)$-type PICF expansions of an irrational number?
\end{question} 

In this paper, we give partial answers to this question as follows. 
We assume that $N=1$ since we also consider an application to the solutions of the Pell equations.

Set
\begin{align*}
&m_1(t):=t^2+1, \\
&m_{2}(s,t):=s^2t^2+t, \\
&m'_{2}(s,t):=s^2t^2+2t, \\
&m_3(s,t):=16t^2s^4 + 8ts^3 + (8t^2 + 1)s^2 + 6ts + t^2 + 1.
\end{align*}
Then, for each $l=1,2,3$, we give a necessary and sufficient condition for existing PICF expansions of square roots of positive square-free integers.
\begin{theorem}\label{iffcondition}
Let $m$ be a positive integer and $l$ is $1,2$ or $3$.
Then $\sqrt{m}$ has a $(1,l)$-type PICF expansion if and only if $m$ is $m_l(s, t)$ or $m'_l(s, t)$ for some $s, t\in \Z$.
\end{theorem}

We also obtain all $(1,l)$-type PICF expansions of square roots of $m_l(s,t)$ for $l=1,2,3$.
\begin{theorem}\label{Maintheorem1}
For all non-zero integers $s$, $t$ except for $m_l(s,t)\le0$, we have
\begin{align}
&\sgn(t)\sqrt{m_1(t)}=[t,\overline{2t}]\label{type11},\\
&\sgn(st)\sqrt{m_{2}(s,t)}=[st, \overline{2s,2st}],\label{type12-1} \\
&\sgn(st)\sqrt{m'_{2}(s,t)}=[st, \overline{s, 2st}],\label{type12-2} \\
&\sgn(t)\sqrt{m_3(s,t)}=[s+(4s^2+1)t, \overline{2s,2s,2(s+(4s^2+1)t)}].\label{type13}
\end{align}
We further obtain
\begin{align}\label{newPICFexpansion1}
\sgn(t)\sqrt{m_3(0,t)} = [-2+ t, \overline{1,-2,-1+ 2t}] = [-1+ t, \overline{2,-1,1+ 2t}],
\end{align} 
\begin{align}\label{newPICFexpansion2}
\sgn(t)\sqrt{m_3(\pm1,t)} = [2+5t,\overline{-2,3,3+10t}] = [1+5t, \overline{3,-2,3+10t}]
\end{align}
for $t \in \Z\setminus\{0\}$ and $\sqrt{m_3(\pm 1,0)}=\sqrt{2}=[2,\overline{-2,3,3}]=[1,\overline{3,-2,3}]$. 

Moreover, these PICFs are all of $(1,1), (1,2)$ and $(1,3)$-type PICF expansions of square roots of positive square-free integers.
\end{theorem}

Note that (\ref{type11}), (\ref{type12-1}), (\ref{type12-2}) and (\ref{type13}) are immediately obtained from classical results.
Indeed, we obtain them from \cite[{p125, l.13}]{Jacobson-Williams} if $s$ and $t$ are both positive integers.
In addition, we also obtain them by using a convergent algorithm (e.g. \cite[Theorem 4.3]{BEJ}) even if $s$ and $t$ are not necessarily positive integers.   
However, the second part of \cref{Maintheorem1} is not obtained from the above method and the authors could not find any explicit references which contain (\ref{newPICFexpansion1}) and (\ref{newPICFexpansion2}).

To prove \cref{iffcondition,Maintheorem1}, 
we determine all integer points on Periodic Continued Fraction varieties (PCF varieties for short) in \cite{BEJ}. 
Here, a PCF variety is an algebraic variety such that some integer points correspond to PICF expansions of a quadratic irrational number.
We explain the definition of PCF varieties in \S\ref{PCFvariety}.

As an application of \cref{Maintheorem1}, 
we obtain fundamental solutions of some families of Pell equations from PICF expansions of square roots of positive square-free integers.
In general, all integer solutions of the Pell equation $x^2-my^2=\pm 1$ is generated by a solution, called a fundamental solution.
Hence, we are interested in the fundamental solutions of the Pell equations. 
It is well-known that for a non-square positive integer $m$, there is an algorithm of obtaining a fundamental solution of the Pell equation $x^2-my^2=\pm 1$ from the RPCF expansion of $\sqrt{m}$.
On the other hand, the algorithm does not work when we consider PICFs of $\sqrt{m}$ in general. 
Indeed we obtain a solution $(x, y)=(-7, -5)$ of the Pell equation from $\sqrt{m_3(\pm 1, 0)} = \sqrt{2} = [2, \overline{-2, 3, 3}]$ and this is not a fundamental solution $(x,y)=(1,1)$. 
Hence we can consider the following question.

\begin{question}\label{question2}
For every positive integer $m$, when do we obtain a fundamental solution\footnote{In general, a fundamental solution of the Pell equation $(x,y)$ requires that both $x$ and $y$ are positive. 
However, in this paper, we allow that $x, y$ are not necessarily positive. 
For details, see \S\ref{pellequation}.} of the Pell equation $x^2-my^2=\pm 1$ from any $(1,l)$-type PICF expansions of $\sqrt{m}$ for each $l \in \Z_{\geq 2}$ as well as the RPCF expansion of $\sqrt{m}$?
\end{question}

In other words, 
this question asks whether a fundamental solution of the Pell equation $x^2-my^2=\pm 1$ is obtained from the $(l-1)$th convergent of a PICF expansion of $\sqrt{m}$ for each $l \in \Z_{\geq 2}$.
Here, the $n$th convergent of a continued fraction $ [c_1, c_2, \dots]$ is $[c_1,c_2,\dots, c_n]$.

If $l=1$, the answer of this question is obtained from a classical result.
Indeed, it is a classical result that $(x,y)=(t,1)$ is a fundamental solution of the Pell equation $x^2-m_1(t)y^2=\pm 1$ for every non-zero integer $t$ and we can obtain it from all PICF expansions of $\sqrt{m_1(t)}$ given in \cref{Maintheorem1}. 
In this paper, we answer this question completely for $l=2,3$.  
\begin{theorem}\label{Maintheorem2-2}
\begin{enumerate}
    \item Suppose that $s,t$ are non-zero integers with $m_2(s,t)> 0$. 
    Then, for each $s,t$, a fundamental solution of the Pell equation $x^2-m_2(s,t)y^2=\pm 1$ is NOT obtained from the $1$st convergent of the PICF expansion of $\sqrt{m_2(s,t)}$ given in \cref{Maintheorem1} if and only if $\abs{s}\geq 2$ and $t=-1$.
    \item Suppose that $s,t$ are non-zero integers with $m_3(s,t)> 0$ (resp.\ $m'_2(s,t)>0$).
    Then, for each $s,t$, a fundamental solution of the Pell equation $x^2-m_3(s,t)y^2=\pm 1$ (resp.\ $x^2-m'_2(s,t)y^2=\pm 1$) is obtained from the $2$nd (resp.\ $1$st) convergent of the PICF expansion of $\sqrt{m_3(s,t)}$ (resp.\ $\sqrt{m'_2(s,t)}$) given in \cref{Maintheorem1}.
\end{enumerate}
\end{theorem}
Note that for each $s,t$, we can obtain a fundamental solution of the Pell equation $x^2-m_2(s,t)y^2=\pm 1$ from the $0$th convergent of the PICF expansion of $\sqrt{m_2(s,t)}$ if $\abs{s}\geq 2$ and $t=-1$. 

As a by-product of the proof of \cref{Maintheorem2-2}, 
we also find fundamental solutions of some families of the Pell equations.

\begin{corollary}\label{Maintheorem2}
Let $s,t$ be non-zero integers.
\begin{enumerate}
\item If $s,t$ satisfy $m_2(s,t)>0$ and $t\neq -1$, 
then 
\begin{align*}
(x_{2}(s,t), y_{2}(s,t)):=(2s^2t+1, 2s)   
\end{align*}
is a fundamental solution of the Pell equation $x^2-m_2(s,t)y^2=\pm 1$.

\item If $s,t$ satisfy $m'_2(s,t)>0$, 
then 
\begin{align*}
(x'_{2}(s,t), y'_{2}(s,t)):= (s^2t+1, s)   
\end{align*}
is a fundamental solution of the Pell equation $x^2-m'_2(s,t)y^2=\pm 1$.

\item 
$(x_{3}(s,t), y_{3}(s,t)):=(16ts^4 + 4s^3 + 8ts^2 + 3s + t, 4s^2+1)$
is a fundamental solution of the Pell equation $x^2-m_3(s,t)y^2=\pm 1$.
\end{enumerate}
Moreover, $(x_3(s,t), y_3(s,t))$ is also a fundamental solution of the Pell equation $x^2-m_3(s,t)y^2=\pm 1$ for each pair of integers $(s,t)$ except for $(s,t)=(\pm1,0), (0,0)$. 
\end{corollary}

Of course, 
\cref{Maintheorem2} are classical results when $s, t$ are positive integers.
Indeed, we obtain them by applying the algorithm of obtaining fundamental solutions of the Pell equations to the first part of \cref{Maintheorem1}.
In this theorem, we claim that these results also hold when $s, t$ are not necessarily positive.

We also note that \cref{Maintheorem2-2} can be regarded as the results for fundamental solutions of some families of the Pell equations parametrized by non-zero integers $s,t$.
There are many related results about fundamental solutions of them  (cf. \cite{McLaughlin2003-2, Nathanson, Mollin1997, Mollin2001, Mollin2003, Mollin-Goddard, Ramasamy1994}). 
However, these previous results do not seem to cover our results even if fundamental solutions look like our results since the range of parameter values is different.

The contents of this article is as follows. 
In \S\ref{Preliminaries}, we recall fundamental solutions of the Pell equations and PCF varieties. 
In \S\ref{Integerpoints}, we determine the set of integer points on some PCF varieties. 
This is a key ingredient when we prove \cref{Maintheorem1}. 
In \S\ref{proof}, we prove \cref{iffcondition,Maintheorem1,Maintheorem2,Maintheorem2-2}. 
In \S \ref{Z2extension}, 
instead of $\sqrt{m}$,
we consider PCF expansions of certain algebraic integers related to the $\Z_2$-extension over $\Q$.
Moreover, we also discuss the relationship between such PCFs and the generalized Pell equations.


\section{Preliminaries}\label{Preliminaries}
In this section, 
we recall fundamental solutions of the Pell equations and PCF varieties. 
We will use them in the proof of \cref{iffcondition,Maintheorem1,Maintheorem2}.

\subsection{Pell equation}\label{pellequation}

In this subsection, 
we recall the classical algorithm for finding a fundamental solution of the Pell equation.
Through this section, let $m$ be a positive non-square integer.
We consider all integer solutions of the Pell equation $x^2-my^2=\pm 1$. 
Set
\begin{align*}
W:=\{(u,v)\in \Z^2 \mid \text{$u^2-mv^2=1$ or $u^2-mv^2=-1$}\}.   
\end{align*}
There is a natural bijection $(u,v)\mapsto u+v\sqrt{m}$ between $W$ and the group of units of the ring $\Z[\sqrt{m}]$. 
Moreover the group of units is isomorphic to $\Z/2\Z\oplus \Z$ by Dirichlet's unit theorem.
Under these bijections, we define a fundamental solution of the Pell equation $x^2-my^2=\pm 1$.
\begin{definition}
We call $(u,v)\in W$ fundamental if $(u,v)$ corresponds to one of $(1,1), (1,-1), (0,1), (0,-1) \in \Z/2\Z\oplus\Z$. 
\end{definition}
Note that an isomorphism between the group of units $\Z[\sqrt{m}]$ and $\Z/2\Z\oplus \Z$ is not canonical. 
However, for any isomorphism, the corresponding elements to $\{(1, \pm1), (0, \pm1)\}$ are the same. 
Thus the fundamental solutions are well-defined.

To find a fundamental solution of the Pell equation $x^2-my^2=\pm 1$, 
we can use the RPCF expansion of $\sqrt{m}$.
To explain this explicitly, we recall some facts and definitions. 
First, it is well-known that $\sqrt{m}$ has the $(1,l)$-type RPCF expansion for some $l \in \Z_{\geq 1}$ (cf. \cite[Theorem 7.21]{NZM}). 
Hence we can write $\sqrt{m}=[a_0,\overline{a_1,\dots, a_l}]$ for some $a_0, a_1, \dots, a_l\in \Z_{\geq 1}$ where $l$ is the period.
Next, we define the convergent of a RPCF $[a_0,a_1,a_2,\dots]$.
\begin{definition}
We define the $n$-th convergent $p_n/q_n$ of the RPCF $[a_0,a_1,a_2,\dots]$ as  
\[
(p_n,q_n)=(a_np_{n-1}+p_{n-2},a_nq_{n-1}+q_{n-2})
\]
for each $n\geq 1$.
Here, $(p_{-1}, q_{-1}):=(1, 0)$ and $(p_0, q_0):=(a_0, 1)$. 
\end{definition}
Note that $p_n/q_n=[a_0,\dots,a_n]$ (cf. \cite[Theorem 7.4]{NZM}).

Under the above preparation, the following holds.
\begin{proposition}\label{fundamentalsolution}
A fundamental solution of the Pell equation $x^2-my^2=\pm1$ is given by
\[
(x,y)=(p_{l-1},q_{l-1}).
\]
where $l$ is the period of the RPCF expansion of $\sqrt{m}$.
\end{proposition}
For the proof of \cref{fundamentalsolution}, see \cite[Theorem 7.25]{NZM}.

\subsection{PCF variety}\label{PCFvariety}
In this subsection,
we introduce the definition of PCF varieties over a number field $K$.
Note that we define only $(N,l)$-type PCF varieties of square roots of $x^2-\alpha=0$ for $\alpha\in K$ which are the needed ones in our paper. 
See \cite[Section 3]{BEJ} for the definition of general PCF varieties.

Before we introduce the definition of PCF varieties, 
we prepare for some notations. 
For $a \in \C$ define 
\begin{align*}
D(a):=\begin{bmatrix}
a & 1 \\
1 & 0 \\
\end{bmatrix}.
\end{align*}
For a finite RCF $[c_1, \dots, c_n]$,
define
\begin{align*}
M([c_1,\dots,c_n]) =
\begin{bmatrix}
M([c_1,\dots,c_n])_{11} & M([c_1,\dots,c_n])_{12} \\
M([c_1,\dots,c_n])_{12} & M([c_1,\dots,c_n])_{22} \\
\end{bmatrix}
:= D(c_1)D(c_2)\cdots D(c_n). 
\end{align*}
For $N\in\Z_{\geq 0}$ and $l\in\Z_{\geq 1}$, we set 
\begin{align*}
&E((y_1,\dots, y_N, x_1,\dots, x_l))\\
&=
\begin{bmatrix}
E((y_1,\dots,y_N,x_1,\dots,x_l))_{11} & E((y_1,\dots,y_N,x_1,\dots,x_l))_{12} \\
E((y_1,\dots,y_N,x_1,\dots,x_l))_{21} & E((y_1,\dots,y_N,x_1,\dots,x_l))_{22} \\
\end{bmatrix}\\
&:=
\begin{cases}
M([y_1,\dots,y_N,x_1,\dots,x_l,0,-y_N,\dots,-y_1,0]) & \text{if $N\geq 1$,}\\
M([x_1,\dots,x_l]) & \text{if $N=0$}. 
\end{cases}
\end{align*}
Note that we denote $E((x_1,\dots, x_l))$ when $N=0$.

\begin{definition}
Fix $N\in\Z_{\geq 0}$ and $l \in \Z_{\geq 1}$.
For $\alpha\in K$,
we define a PCF variety of $(N, l)$-type by the equations
\begin{align*}
\begin{cases}
E((y_1,\dots,y_N,x_1,\dots,x_l))_{22}-E((y_1,\dots,y_N,x_1,\dots,x_l))_{11}=0,\\
E((y_1,\dots,y_N,x_1,\dots,x_l))_{12}=-\alpha E((y_1,\dots,y_N,x_1,\dots,x_l))_{21},
\end{cases}
\end{align*}
where $y_1,\dots,y_N, x_1, \dots, x_l$ are variables.
We write this variety as $V(\alpha)_{N, l}$.
In what follows, $(y_1,\dots,y_N,x_1,\dots,x_l)$ denotes the coordinate of $V(\alpha)_{N, l}$.
\end{definition}

In what follows,
we consider $N=1$ and $\alpha=\sqrt{m}$ where $m$ is a positive square-free integer.
All $(1,l)$-type PICF expansions of $\sqrt{m}$ come from integer points on $V(\sqrt{m})_{1, l}$, 
that is, the following proposition holds.
\begin{proposition}{(\cite[(a) in Section 3.1]{BEJ})}\label{necessarycondition}
If $\sqrt{m}$ has a $(1,l)$-type PICF expansion $[b_1, \overline{a_1, \dots, a_l}]$, 
then $(b_1, a_1, \dots, a_l)\in \A^{1+l}$ is an integer point on $V(\sqrt{m})_{1,l}$. 
\end{proposition}
Hence, we should determine all elements of $V(\sqrt{m})_{1,l}(\Z)$ if we determine all $(1,l)$-type PICF expansions of $\sqrt{m}$.
Moreover, the converse of \cref{necessarycondition} also holds, that is, 
for $(b_1, a_1,\dots, a_l) \in V(\sqrt{m})_{1,l}(\Z)$, we obtain $\pm\sqrt{m}=[b_1, \overline{a_1,\dots, a_l}]$ if RHS converges (cf.\ \cite[Proposition 2.8]{BEJ}).\footnote{
More precisely, if $[b_1, \overline{a_1,\dots, a_l}]$ converges, then $[b_1, \overline{a_1,\dots, a_l}]$ is $\sqrt{m}$ or $-\sqrt{m}$.
However, we can determine the sign of $[b_1, \overline{a_1,\dots, a_l}]$ (cf. \cite[Theorem 4.3]{BEJ}). 
} 
This also plays an important role when we prove main theorems.

We also remark that if $a_1\cdots a_l = 0$, then the period of $[b_1, \overline{a_1, \dots, a_l}]$ is less than $l$ by \cite[Lemma 2.2]{BEJ}. 
Hence we only consider non-degenerate integer points on $V(\sqrt{m})_{1,l}(\Z)$, defined as follows.

\begin{definition}
Let $V(\sqrt{m})_{1, l}$ be a PCF variety.
An integer point on $V(\sqrt{m})_{1, l}$ $(b_1, a_1, \dots, a_l)$ is said to be non-degenerate if $a_i \neq 0$ for all $1\leq i\leq l$.
We write the set of non-degenerate integer points on $V(\sqrt{m})_{1,l}$ as $V(\sqrt{m})_{1,l}(\Z)_{\text{nd}}$. 
\end{definition}

\begin{remark}
There are some results about geometric properties of PCF varieties (\cite{JZ,JLZ}).
In particular, $V(\sqrt{m})_{1,l}(\Z)_{\text{nd}}$ is a finite set from the proof of \cite[Theorem 2.5]{JLZ} for $l \leq 3$. 
\end{remark}


\section{Integer points on PCF varieties}\label{Integerpoints}

In this section,
we determine all non-degenerate integer points on PCF varieties of $(1, l)$-type for each $l=1,2,3$.
We recall that
\begin{align*}
&m_1(t)=t^2+1, \\
&m_{2}(s,t)=s^2t^2+t, \\
&m'_{2}(s,t)=s^2t^2+2t, \\
&m_3(s,t)=16t^2s^4 + 8ts^3 + (8t^2 + 1)s^2 + 6ts + t^2 + 1.
\end{align*}


\subsection{$(1,1)$-type}

We see that $E((y_1,x_1))$ is
\begin{align*}
\begin{bmatrix} 
y_1& x_1y_1+1-y_1^2\\
1 & x_1-y_1
\end{bmatrix}.
\end{align*}
Hence $V(\sqrt{m})_{1,1}$ is given by
\begin{align*}
\begin{cases}
x_1-2y_1 = 0,\\
y_1^2-x_1y_1-1 = -m.
\end{cases} 
\end{align*}
By easy calculation, $V(\sqrt{m})_{1,1}$ consists of two points $\pm(\sqrt{m-1}, 2\sqrt{m-1})$. 
Therefore, we immediately obtain the following propositions (cf.\ \cite[Proposition 5.3]{BEJ}).

\begin{proposition}\label{11iffcondition}
$V(\sqrt{m})_{1,1}(\mathbb{Z})$  has a non-degenerate integer point if and only if $m=m_1(t)$ for some $t\in \Z\setminus\{0\}$.
\end{proposition}

\begin{proposition}\label{11effective}
Suppose that $m=m_1(t)$ for some $t\in \Z\setminus\{0\}$.
Then, 
\begin{align*}
V(\sqrt{m})_{1,1}(\Z)_{\text{nd}}= \{\pm (t, 2t)\}.  
\end{align*}
\end{proposition}


\subsection{$(1,2)$-type}

We see that $E((y_1,x_1,x_2))$ is 
\begin{align*}
\begin{bmatrix} 
y_1x_1+1& y_1x_1x_2+x_2-y_1^2x_1\\
x_1 & x_1x_2-y_1x_1+1
\end{bmatrix}.
\end{align*}
Hence $V(\sqrt{m})_{1,2}$ is given by
\begin{numcases}{}
x_1x_2 - 2y_1x_1 = 0\label{121stequation}\\
y_1^2x_1-y_1x_1x_2-x_2=-mx_1\label{122ndequation}
\end{numcases}

\begin{proposition}\label{12iffcondition}
$V(\sqrt{m})_{1,2}$ has a non-degenerate integer point if and only if $m=(st/2)^2 + t$ for some $s, t\in \Z\setminus\{0\}$ with $2\mid st$.
\end{proposition}

\begin{proof}
Suppose that $m=(st/2)^2 + t$ for some $s, t\in \Z\setminus\{0\}$ with $2\mid st$.
Then, we find $(st/2, s, st)\in V(\sqrt{m})_{1,2}(\Z)_{\text{nd}}$ and we can check the if part.

Hence it is sufficient to show the only if part.
Suppose that $(y_1, x_1, x_2)\in V(\sqrt{m})_{1,2}$ is a non-degenerate integral point. 
From (\ref{121stequation}), we obtain $x_2=2y_1$ since $x_1 \neq 0$.
By substituting it into (\ref{122ndequation}), we obtain $y_1^2x_1+2y_1-mx_1=0$ and $m=y_1^2+2y_1/x_1$. 
Since $m \in \Z$, $2y_1=tx_1$ for some $t \in \Z\setminus\{0\}$.
Hence, we obtain $m=(tx_1/2)^2+t$ and complete the proof by putting $s=x_1$. 
\end{proof}

\begin{remark}\label{m2modify}
Since $s$ or $t$ is even from $2\mid st$, 
the condition $m=(st/2)^2 + t$ for some $s, t\in \Z\setminus\{0\}$ with $2\mid st$ is equivalent to $m = m_2(s,t) \ \text{or} \ m'_2(s,t)$ for some $s,t\in \Z\setminus\{0\}$.  
\end{remark}

From the proof of \cref{12iffcondition} and \cref{m2modify}, 
we obtain the following corollary.

\begin{corollary}\label{12effective}
\begin{align*}
&V(\sqrt{m})_{1,2}(\Z)_{\text{nd}}\\
&=\{\pm(st, 2s, 2st) \mid \text{$m = m_2(s,t)$, $s,t\neq 0$}\}\cup\{\pm(st,s,2st) \mid \text{$m=m'_2(s,t)$, $s,t\neq 0$}\}.
\end{align*}
\end{corollary}


\subsection{$(1,3)$-type}

We see that $E((y_1,x_1,x_2,x_3))$ is
\begin{align*}
\begin{bmatrix} 
y_1x_2x_1 + (x_2 + y_1) & ((y_1x_3 - y_1^2)x_2 + y_1)x_1 + (x_3 - y_1)x_2 + y_1x_3 - y_1^2 + 1\\
x_2x_1 + 1   &  ((x_3 - y_1)x_2 + 1)x_1 + x_3 - y_1
\end{bmatrix}.
\end{align*}
Hence $V(\sqrt{m})_{1,3}$ is given by
\begin{numcases}{}
2y_1x_2x_1+2y_1-x_3x_2x_1+x_2-x_1-x_3 = 0,\label{1st13equation}\\
m(x_2x_1+1) = y_1(x_3x_2x_1+x_1+x_3-x_2)-y_1^2(x_2x_1+1)+x_3x_2+1.\label{2nd13equation}
\end{numcases}

Before we determine non-degenerate integral points on $V(\sqrt{m})_{1,3}$, 
we give a necessary and sufficient condition for existing non-degenerate integral points on $V(\sqrt{m})_{1,3}(\Z)$.

\begin{proposition}\label{13iffcondition}
$V(\sqrt{m})_{1,3}$ has a non-degenerate integral point if and only if
$m= m_3(s,t)$ for some $s, t\in \Z$ with $(s,t)\neq(0,0)$.
\end{proposition}

\begin{proof}
    Suppose that $m=16t^2s^4+8ts^3+(8t^2+1)s^2+6ts+t^2+1$ for some $s,t\in \Z$ with $(s,t)\neq (0,0)$.
    Then, we can take non-degenerate integral points such as $(s+(4s^2+1)t, 2s, 2s, 2(s+(4s^2+1)t))$ and we showed the if part.
    
    Hence, it is sufficient to show the only if part.
    Suppose that $(y_1,x_1,x_2,x_3)\in V(\sqrt{m})_{1,3}(\Z)_{\text{nd}}$. 
    Then, from (\ref{1st13equation}), we obtain
    \begin{align}\label{firstequation}
    \abs{(2y_1-x_3)(x_2x_1+1)} = \abs{x_1-x_2}.
    \end{align} 
    In the following, we divide the proof into two cases, 
    $\abs{x_1-x_2} > \abs{x_2x_1+1}$ and $\abs{x_1-x_2} \leq \abs{x_2x_1+1}$.

    \begin{enumerate}
    \item If $\abs{x_1-x_2} > \abs{x_2x_1+1}$, 
    we obtain
    \[
    \begin{cases}
    \frac{-x_1-1}{x_1-1} < x_2 < \frac{x_1-1}{x_1+1} & \text{if $x_1 \neq \pm 1$},\\
    x_2>0 & \text{if $x_1=-1$},\\
    x_2<0 & \text{if $x_1=1$}.
    \end{cases}
    \]
    Hence, it is sufficient to consider the following six cases;
    \begin{align*}
    &(x_1, x_2) = \pm(2,-2), 
    \text{$x_1\geq 2$ and $x_2=-1$},
    \text{$x_1\leq -2$ and $x_2=1$},\\
    &\text{$x_1=-1$ and $x_2>0$}, 
    \text{$x_1=1$ and $x_2<0$}.
    \end{align*}
    
    If $(x_1, x_2) = \pm (2,-2)$, we immediately show that $V(\sqrt{m})_{1,3}(\Z)_{\text{nd}}= \emptyset$.
    Hence, we consider the left cases. 
    \begin{enumerate}
        \item If $\abs{x_1} \geq 2$, 
        consider the case $x_1 \geq 2$ and $x_2=-1$. 
        By substituting $x_2=-1$ into (\ref{1st13equation}),
        we obtain $(x_1-1)(x_3-2y_1-1)=2$ and $(x_1, x_3)=(3, 2y_1+2), (2, 2y_1+3)$ since $x_1\geq 2$.
        If $(x_1, x_3)=(3, 2y_1+2)$, there are no non-degenerate integral points since $y_1\not\in \Z$. 
        If $(x_1, x_3)=(2, 2y_1+3)$, we obtain $y_1^2+2y_1+2-m=0$ by (\ref{2nd13equation}) and $y_1=-1\pm \sqrt{m-1}$. 
        Since $y_1 \in \Z$, the necessary condition for $V(\sqrt{m})_{1,3}(\Z)_{\text{nd}}\neq\emptyset$ is $m=t^2+1$ for some $t\in \Z\setminus\{0\}$.  
    
        Similarly, the necessary condition for $V(\sqrt{m})_{1,3}(\Z)_{\text{nd}}\neq\emptyset$ is $m=t^2+1$ for some $t\in \Z\setminus\{0\}$ when $x_1 \leq -2$, $x_2 = 1$. 
    
        \item If $\abs{x_1}=1$, 
        consider the case $x_1=-1$ and $x_2 > 0$. 
        By substituting $x_1=-1$ into (\ref{1st13equation}),
        we obtain $(x_2-1)(x_3-2y_1+1)=-2$ and $ (x_2, x_3)=(2,2y_1-3), (3,2y_1-2)$ since $x_2> 0$.
        If $(x_2, x_3)=(3,2y_1-2)$, there are no non-degenerate integral points since $y_1\not\in \Z$.
        If $(x_2, x_3)=(2, 2y_1-3)$, we obtain $y_1^2-4y_1+5-m=0$ from (\ref{2nd13equation}) and $y_1=2\pm\sqrt{m-1}$.
        Since $y_1 \in \Z$, the necessary condition for $V(\sqrt{m})_{1,3}(\Z)_{\text{nd}}\neq\emptyset$ is $m=t^2+1$ for some $t\in \Z\setminus\{0\}$. 
    
        Similarly, 
        the necessary condition for $V(\sqrt{m})_{1,3}(\Z)_{\text{nd}}\neq\emptyset$ is $m=t^2+1$ for some $t\in \Z\setminus\{0\}$ when $x_1=1, x_2 < 0$.  
    \end{enumerate}
    
    \item If $\abs{x_1-x_2} \leq \abs{x_2x_1+1}$,
    we obtain $(\abs{2y_1-x_3}-1)\abs{x_2x_1+1} \leq 0$.
    In what follows, we assume  $\abs{2y_1-x_3}-1\leq 0$ since we can check that there are no non-degenerate integral points if $\abs{2y_1-x_3}-1>0$\footnote{In this case, we obtain $x_2x_1=-1, x_2=x_1$ but there are no integer $x_2,x_1$ satisfying the equations.}. 
    
    \begin{enumerate}
    \item If $\abs{2y_1-x_3}-1=0$, 
    consider the case $2y_1-x_3=1$. 
    From (\ref{1st13equation}),  we obtain $(x_1+1)(x_2-1)=-2$ and $(x_1, x_2)=(-3, 2), (-2,3)$. 
    If $(x_1, x_2)=(-3,2)$, we obtain $5y_1^2-4y_1+1-5m=0$ from (\ref{2nd13equation}) and $y_1 = (2\pm \sqrt{25m-1})/5$.
    Since $y_1\in \Z$, the necessary condition for $V(\sqrt{m})_{1,3}(\Z)_{\text{nd}}\neq\emptyset$ is $m=25t^2+14t+2$ for some $t \in \Z$.
    If $(x_1, x_2)=(-2,3)$, we obtain $5y_1^2-6y_1+2-5m=0$ from (\ref{2nd13equation}) and $y_1 = (3\pm \sqrt{25m-1})/5$.   
    Since $y_1\in \Z$, the necessary condition for $V(\sqrt{m})_{1,3}(\Z)_{\text{nd}}\neq\emptyset$ is $m=25t^2+14t+2$ for some $t \in \Z$.
    
    Similarly, the necessary condition for $V(\sqrt{m})_{1,3}(\Z)_{\text{nd}}\neq\emptyset$ is $m=25t^2+14t+2$ for some $t \in \Z$ when $2y_1-x_3=-1$. 
        
    \item If $\abs{2y_1-x_3}-1 < 0$\label{thirdcase}, 
    we obtain $2y_1=x_3$ from (\ref{firstequation}).
    By substituting it into (\ref{1st13equation}), we obtain $x_1=x_2$. 
    By combining this and (\ref{2nd13equation}), we obtain 
    \begin{align}\label{thirdcase}
    (y_1^2-m)x_1^2+2y_1x_1+y_1^2+1-m=0.
    \end{align}
    Hence, $2 \mid x_1$ and we can write $x_1=2s$ for some $s\in \Z\setminus\{0\}$. 
    By combining (\ref{thirdcase}), $m=y_1^2+\frac{4sy_1+1}{4s^2+1}\in \Z$ and $4sy_1 \equiv -1 \bmod{4s^2+1}$.
    Thus, we obtain $y_1=s+(4s^2+1)t$ for some $t \in \Z$ and the necessary condition for $V(\sqrt{m})_{1,3}(\Z)_{\text{nd}}\neq\emptyset$ is $m=16t^2s^4+8ts^3+(8t^2+1)s^2+6ts+t^2+1$ for some $s\in \Z\setminus \{0\}$ and $t\in \Z$.  
    \end{enumerate}
    \end{enumerate}
    Therefore, combining the results of (1) and (2),
    we showed that if $V(\sqrt{m})_{1,3}(\Z)_{\text{nd}}\neq\emptyset$ then $m=16t^2s^4+8ts^3+(8t^2+1)s^2+6ts+t^2+1=m_3(s,t)$ for some $s,t \in\Z$ with $(s,t)\neq (0,0)$ since $t^2+1=m_3(0,t)$ and $25t^2+14t+2=m_3(\pm 1,t)$.
\end{proof}

By tracing the proof of the only if part of \cref{13iffcondition} and determining $(y_1,x_1,x_2,x_3)$ for each case,
we obtain the following proposition.

\begin{corollary}\label{13effective}
\begin{align*}
&V(\sqrt{m})_{1,3}(\Z)_{\text{nd}}\\
&=\{\pm(s+(4s^2+1)t, 2s, 2s, 2(s+(4s^2+1)t))\mid m=m_3(s,t), s,t \neq 0 \}\\
&\cup\{\pm(-2+t, 1, -2, -1+2t), \pm(-1+t, 2, -1, 1+2t)\mid m=m_3(0,t), t\neq 0\}\\
&\cup\{\pm(2+5t, -2, 3, 3+10t), \pm(1+5t, 3, -2, 3+10t) \mid m=m_3(\pm 1,t)\}.
\end{align*}
\end{corollary}


\section{Proof of the main theorems}\label{proof}

\subsection{The proof of \cref{iffcondition}}\label{proof of Maintheorem1}
First, the only if part follows from \cref{necessarycondition} and the only if part of \cref{11iffcondition}(resp.\ \cref{12iffcondition}, \cref{13iffcondition}) in $l=1$(resp.\ $l=2$, $l=3$). 
To show the if part, it is sufficient to check that $[b_1,\overline{a_1,\dots,a_l}]$ converge for some $(b_1, a_1, \dots, a_l) \in V(\sqrt{m})_{1,l}(\Z)_{\text{nd}}$ by the converse of \cref{necessarycondition}.
Indeed, if a PICF $[b_1,\overline{a_1,\dots,a_l}]$ converge to $-\sqrt{m_l(s,t)}$, we just take $(-b_1,-a_1,...,-a_l)$ instead of $(b_1,a_1,...,a_l)$.
For the convergence, we check the conditions of \cite[Theorem 4.3]{BEJ}.
If $l=1$, $[t,\overline{2t}]$ converges.
If $l=2$, $[st, \overline{2s, 2st}]$ and $[st, s, 2st]$ converge.
If $l=3$, $[s+(4s^2+1)t,\overline{2s, 2s, 2(s+(4s^2+1)t)}]$, $[-2+t, \overline{1,-2,-1+2t}]$, $[-1+t, \overline{2, -1, 1+2t}]$, $[2+5t, \overline{-2,3,3+10t}]$ and $[1+5t, \overline{3,-2,3+10t}]$ converge.
Here, we only mentioned the results of convergence.
For more details, see \cref{appendix}.
Combining these results with \cref{11effective,12effective,13effective},
we complete the proof. 


\subsection{Proof of \cref{Maintheorem1}}\label{proofofMaintheorem1}

Since convergence of PICFs in \cref{Maintheorem1} has already shown in \cref{proof of Maintheorem1},
we  should only determine the signs of PICFs.
For a $P:=(b_1,a_1,\cdots,a_l)$, let $\lambda(P)_+$ be the eigenvalue of $E(P)$ such that $\abs{\lambda(P)_+}\geq 1$.
By \cite[Theorem 4.3]{BEJ}, we obtain that
\[
[b_1,\overline{a_1,...,a_l}]=\frac{\lambda(P)_+-E(P)_{22}}{E(P)_{21}}.
\]
Hence, it is sufficient to determine $\lambda(P)_+$ for each PICF in \cref{Maintheorem1}.
Easy calculations show that
\begin{align*}
&\lambda((t,2t))_+=t+\sgn(t)\sqrt{m_1(t)},\\
&\lambda((st,2s,2st))_+=2s^2t+1+\sgn(st)2s\sqrt{m_2(s,t)},\\
&\lambda((st,s,2st))_+=s^2t+\sgn(st)s\sqrt{m_2'(s,t)},\\
&\lambda((s+(4s^2+1)t,2s,2s,2(s+(4s^2+1)t)))_+\\
&\quad =16ts^4+4s^3+8ts^2+3s+t+\sgn(t)(4s^2+1)\sqrt{m_3(s,t)},\\
&\lambda((-2+t,1,-2,-1+2t))_+=-t-\sgn(t)\sqrt{m_3(0,t)},\\
&\lambda((-1+t,2,-1,1+2t))_+=-t-\sgn(t)\sqrt{m_3(0,t)},\\
&\lambda((2+5t,-2,3,3+10t))_+=-25t-7-\sgn(t)\sqrt{m_3(\pm1,t)},\\
&\lambda((1+5t,3,-2,3+10t))_+=-25t-7-\sgn(t)\sqrt{m_3(\pm1,t)},
\end{align*}
and we complete the proof (for explicit representations of $E(P)$ for each PICF in \cref{Maintheorem1}, see \cref{appendix}).


\subsection{Proof of \cref{Maintheorem2-2,Maintheorem2}}

Before we prove \cref{Maintheorem2-2,Maintheorem2},
we show the following lemma which gives the RPCF expansion of $m_2(s,t)$, $m'_2(s,t)$ and $m_3(s,t)$.

\begin{lemma}\label{npcf}
For $t<0$, we have
\begin{equation}\label{npcf2-1}
\sqrt{m_{2}(s,t)}=
\begin{cases}
[-st-1,\overline{1,2s-2,1,2(-st-1)}] & \text{if $s\geq 2$}, \\
[-t-1,\overline{2,-2t-2}] & \text{if $s=1, t\neq-1$},
\end{cases}
\end{equation}
and
\begin{equation}\label{npcf2-2}
\sqrt{m'_{2}(s,t)}=
\begin{cases}
[-st-1,\overline{1,s-2,1,2(-st-1)}] & \text{if $s\geq 3$}, \\
[-2t-1,\overline{2,2(-2t-1)}] & \text{if $s=2$}, \\
[-t-2,\overline{1,2(-t-2)}] & \text{if $s=1, t\neq-1, -2$} \\
\end{cases}
\end{equation}
If $t>0$ and $s<0$, then we have
\begin{equation}\label{npcf3}
\sqrt{m_3(s,t)}=[s+(4s^2+1)t-1,\overline{1,-2s-1,-2s-1,1,2(s+(4s^2+1)t-1)}].
\end{equation}
\end{lemma}

This lemma is proved as same as \cref{proofofMaintheorem1}. 
Remark that We except the case $(s,t)=(1,-1)$ for $m_{2}$ and $(s,t)=(1,-1), (1,-2)$ for $m'_{2}$ since $m_{2}(1,-1)=0$, $m'_{2}(1,-1)=-1$, and $m'_{2}(1,-2)=0$.

\begin{proof}[Proof of \cref{Maintheorem2-2}]
Suppose that the case $m_2(s,t)$ in the Pell equation. 
We may assume that $s>0$ since $m_{2}(s,t)=m_{2}(-s,t)$.
To prove the if part, consider that $s\geq 2$ and $t=-1$.
Then, we obtain the RPCF expansion $\sqrt{m_2(s,-1)} = [s-1,\overline{1,2s-2}]$ by \cref{npcf}. 
Since the $1$st convergent of $[s-1,\overline{1,2s-2}]$ is $s$, 
a fundamental solution of $x^2-m_2(s,t)y^2=1$ is $(x,y)=(s,1)$ by \cref{fundamentalsolution}. 
However, this is not obtained from PICF expansions of $\sqrt{m_2(s,t)}=[st,\overline{2s,2st}]$ since the $1$st convergent of it is $(-2s^2+1)/(2s)$.

Now we will prove the only if part.
If $t>0$, PICF expansions $[st, \overline{2s,2st}]$ give the RPCF expansions of $\sqrt{m_{2}(s,t)}$.
Hence, fundamental solutions are the $1$st convergents of them by \cref{fundamentalsolution}. 
If $t<0$, 
the RPCF expansions of $\sqrt{m_{2}(s,t)}$ are given in \cref{npcf}. 
Applying \cref{fundamentalsolution} to (\ref{npcf2-1}),
fundamental solutions are the $3$rd convergents of (\ref{npcf2-1}) if $s\geq 2, t\neq -1$ and the $1$st convergents of (\ref{npcf2-1}) if $s=1, t\neq -1$ or $s\geq 2, t=-1$.
In particular, if $s\geq 2, t\neq -1$ and $s=1, t\neq -1$, 
we can check these fundamental solutions coincide with the $1$st convergents of $\sqrt{m_{2}(s,t)}=[-st, \overline{-2s,-2st}]$ up to signs.

Next, suppose that $m_3(s,t)$ in the Pell equation.
The proof of the if part is clear since the $2$nd convergent of $\sqrt{2}=[2,\overline{-2,3,3}]$ does not coincide with the $0$th convergent of $\sqrt{2}=[1,\overline{2}]$.  
Hence it is sufficient to prove the only if part.
\begin{enumerate}
\item Consider the case $\sgn(t)\sqrt{m_3(s,t)}=[s+(4s^2+1)t, \overline{2s,2s,2(s+(4s^2+1)t)}]$ in \cref{Maintheorem1}.
We may assume that $s>0,t>0$ or $s>0, t<0$ since $m_3(s,t)=m_3(-s,-t)$.
If $s>0, t>0$, this is the RPCF expansion of $\sqrt{m_3}(s,t)$ and fundamental solutions are the $2$nd convergents.
If $s<0, t>0$, the $2$nd convergents of $[s+(4s^2+1)t, \overline{2s,2s,2(s+(4s^2+1)t)}]$ coincide with the $4$th convergents of (\ref{npcf3}) in \cref{npcf} up to signs.
\item Consider the case 
\begin{align*}
\sgn(t)\sqrt{m_3(0,t)} = [-2+ t, \overline{1,-2,-1+ 2t}] = [-1+ t, \overline{2,-1,1+ 2t}].  
\end{align*}
Then, both the $2$nd convergents of $[-2+ t, \overline{1,-2,-1+ 2t}]$ and $[-1+ t, \overline{2,-1,1+ 2t}]$ are $-t/(-1)$ and this is a fundamental solution of $x^2-m_3(0,t)y^2=-1$.
\item Consider the case 
\begin{align*}
\sgn(t)\sqrt{m_3(\pm1,t)} = [2+5t,\overline{-2,3,3+10t}] = [1+5t,\overline{3,-2,3+10t}].
\end{align*}
If $t\neq 0$, the $2$nd convergents of $[2+5t,\overline{-2,3,3+10t}]$ and $[1+5t, \overline{3,-2,3+10t}]$ coincide with the fourth convergents of (\ref{npcf3}) up to signs. 
\end{enumerate}

Finally, suppose that the case $m'_2(s,t)$ in the Pell equation.
We may assume that $s>0$ since $m'_{2}(s,t)=m'_{2}(-s,t)$.
If $t>0$, PICF expansions $[st, \overline{s,2st}]$ give the RPCF expansions of $\sqrt{m'_{2}(s,t)}$. 
Hence, fundamental solutions are the $1$st convergents of them by \cref{fundamentalsolution}. 
If $t<0$, 
the RPCF expansions of $\sqrt{m'_{2}(s,t)}$ are given in \cref{npcf}.
Applying \cref{fundamentalsolution} to (\ref{npcf2-2}), 
fundamental solutions are the $3$rd convergents of (\ref{npcf2-2}) if $s\geq 3$ and the $1$st convergents of (\ref{npcf2-2}) if $s=2$ and $s=1$, $t\neq -1,-2$.
For each case, we can check these fundamental solutions coincide with the $2$nd convergents of $\sqrt{m'_{2}(s,t)}=[-st, \overline{-s, -2st}]$ up to signs. 
\end{proof}

We remark that we obtain \cref{Maintheorem2} by writing down the convergents explicitly in the proof of \cref{Maintheorem2-2}. 


\section{$\Z[X_{n-1}]$-PCF expansions of $X_n$ and the generalized Pell equation}\label{Z2extension}

In this section,
we consider PCF expansions of certain algebraic integers related to the $\Z_2$-extension over $\Q$.
For each non-negative integer $n$, set $X_n=2\cos(\pi/2^{n+1})$.
For example, 
\[
X_0=0, X_1=\sqrt{2}, X_2=\sqrt{2+\sqrt{2}},\dots.  
\]
Then,
$\mathbb{B}_n:=\mathbb{Q}(X_n)$ is the Galois extension over $\Q$ with $\Gal(\mathbb{B}_n/\Q)\cong \Z/2^n\Z$ and the ring of integers of $\mathbb{B}_n$ is $\Z[X_n]$ for $n\geq 0$. 
Note that $\cup_{n\geq 0}\mathbb{B}_n$ is the $\mathbb{Z}_2$-extension over $\mathbb{Q}$.


\subsection{$(0,3)$-type $\Z[X_{n-1}]$-PCF expansion of $X_n$}

For $a_i\in\Z[X_{n-1}]$,
let $\{a_i\}$ be a sequence satisfying the periodic condition, that is, 
$a_k=a_{l+k}$ for some $l \in \Z_{\geq 1}$, $N \in \Z_{\geq 0}$ and all $k \geq N$.
In a similar manner to a PICF,
we call $[a_0,\dots,a_{N-1},\overline{a_{N},\dots,a_{N+l-1}}]$ a $(N,l)$-type $\Z[X_{n-1}]$-PCF.

Block-Elkies-Jordan asked the following question;
\begin{question}{\cite[Problem 1.1]{BEJ}}\label{pcfproblem}
Find $(N,l)$-type $\Z[X_{n-1}]$-PCF expansions of $X_{n}$ for each $n\geq 0$, $N\geq 0$ and $l\geq 1$.
\end{question}

In \cite{BEJ},
Block-Elkies-Jordan gave partial answers for \cref{pcfproblem}.
More precisely,
for $n=1,2$,
they determined the all $(0,1)$, $(0,2)$, $(0,3)$, $(1,1)$, $(1,2)$ and $(2,1)$-types $\Z[X_{n-1}]$-PCF expansions of $X_{n}$.
Moreover, 
they showed that there are no $(0,1)$, $(0,2)$ and $(1,1)$-types $\Z[X_{n-1}]$-PCF expansions of $X_n$ for all $n\geq 2$.
In comparison to this,
Yoshizaki found $(1,2)$-type $\Z[X_{n-1}]$-PCF expansions of $X_n$ for all $n\geq 1$ by using a different approach.
\begin{theorem}{\cite[Theorem 3.4]{Yoshizaki}}
For all $n\geq 1$,
we obtain
\begin{align*}
X_n=\left[1,\overline{\frac{2}{1+X_{n-1}}, 2}\right].
\end{align*}
\end{theorem}

In this paper,
we find $(0,3)$-type $\mathbb{Z}[X_{n-1}]$-PCF expansions of $X_n$ for all $n\geq 1$ by using PCF varieties.
Before we state our result and prove it, 
we prepare some notions and facts.
Let $\Z[X_n]^*$ be the group of units of $\Z[X_n]$ and $\tau_n$ is the generator of ${\rm Gal}(\mathbb{B}_n/\mathbb{B}_{n-1})\cong \mathbb{Z}/2\mathbb{Z}$.
We define the relative norm of the quadratic extensions $\mathbb{B}_n/\mathbb{B}_{n-1}$ by $N_{n/n-1}:\mathbb{B}_n\to \mathbb{B}_{n-1};x\mapsto x\tau_n (x)$.
For $x \in \mathbb{Z}[X_n]$, there exists a unique pair $(a,b) \in \mathbb{Z}[X_{n-1}]^2$ such that $x=a+X_nb$,
and the relative norm is of the form $N_{n/n-1}(x)=a^2-X_n^2b^2$.
Set
\[
  \eta_n=1+\sum_{k=1}^{2^n-1}2\cos\left(\frac{k\pi}{2^{n+1}}\right).
\]
Note that $\eta_n$ satisfies $N_{n/{n-1}}(\eta_n)=-1$ (cf.\ \cite[(6.1)]{MO2020}). 
In particular, $\eta_n \in \Z[X_n]^*$ and we call it Horie unit or Weber's normal unit (cf.\ {\cite{Horie2005,MO2016}}).
Our result is the following.
\begin{theorem}\label{pcfofXn}
For all non-negative integers $n$ and $\sigma\in\Gal(\mathbb{B}_n/\Q)$, 
we have
\[
    \sigma(X_n)=\left[\overline{\sigma\left(\frac{(\eta_n-\eta_{n-1})X_n-1}{\eta_{n-1}}\right),\sigma(\eta_{n-1}),\sigma\left(\frac{(\eta_n-\eta_{n-1})/X_n-1}{\eta_{n-1}}\right)}\right]
\]
up to signs.
\end{theorem}

\begin{proof}
In a similar manner to \cref{PCFvariety},
for $(a_1,a_2,a_3)\in V(X_n)_{0,3}(\Z[X_{n-1}])$, 
we obtain $(0,3)$-type $\Z[X_{n-1}]$-PCF of $X_{n-1}$ if $[\overline{a_1,a_2,a_3}]$ converges.
Hence, 
to prove \cref{pcfofXn},
we find elements of $V(X_n)_{0,3}(\Z[X_{n-1}])$ and check the convergence for $n\geq 1$.
We see that $E((x_1,x_2,x_3))$ is
\[
\begin{bmatrix}
    x_1x_2x_3+x_1+x_3 & x_1x_2+1 \\
    x_2x_3+1 & x_2 \\
\end{bmatrix}
\].

Hence, $V(X_n)_{0,3}$ is given by
\begin{equation}\label{zerothree}
\begin{cases}
x_2-x_1x_2x_3-x_1-x_3=0, \\
x_1x_2+1=X_n^2(x_2x_3+1).
\end{cases}
\end{equation}
By eliminating $x_1$ from (\ref{zerothree}), 
we obtain
\begin{align}\label{norm}
x_2^2-X_n^2(x_2x_3+1)^2=-1.
\end{align}
Here, we take
\begin{equation}\label{solution}
\begin{split}
    x_2&=\eta_{n-1},\\
    x_3&=\left.\left(-1+\sum_{\substack{1\leq k \leq 2^n-1,\\ 2\nmid k}}\epsilon_{k,n}\right)\right/ \eta_{n-1}
\end{split}
\end{equation}
where $\epsilon_{k,n}=(2\cos(k\pi/2^{n+1}))/(2\cos(\pi/2^{n+1}))$.
Then, 
we obtain $\eta_n=x_2+X_n(x_2x_3+1)$ since 
\begin{equation*}
\begin{split}
\eta_n&=1+\sum_{k=1}^{2^n-1}2\cos(k\pi/2^{n+1}) \\
&=\eta_{n-1}+\sum_{\substack{1\leq k \leq 2^n-1,\\ 2\nmid k}}2\cos(k\pi/2^{n+1}) \\
&=\eta_{n-1}+2\cos(\pi/2^{n+1})\sum_{\substack{1\leq k \leq 2^n-1,\\ 2\nmid k}}\epsilon_{k,n}.
\end{split}
\end{equation*}
Combining with $N_{n/{n-1}}(\eta_n)=-1$,
we see that (\ref{solution}) are solutions of (\ref{norm}).
We also obtain $\epsilon_{k,n}\in \Z[X_n]^*$ since $2\cos(k\pi/2^{n+1})$ are prime elements on $2$ in $\mathbb{Z}[X_n]$ for each odd integer $k$ and $2$ ramifies completely in $\mathbb{B}_n/\mathbb{Q}$ for $n\geq 1$.
From $\tau_n(\epsilon_{k,n})=\epsilon_{k,n}$, 
we see that $\epsilon_{k,n}\in \Z[X_{n}]^{*}\cap \mathbb{B}_{n-1}=\Z[X_{n-1}]^{*}$ for each odd integer $k$ and $n\in\Z_{\geq 1}$.
Hence, $x_2,x_3\in\Z[X_{n-1}]$.

By the definition of $X_n$ and $\eta_n$,
we obtain
\begin{equation*}
\sum_{\substack{1\leq k \leq 2^n-1,\\ 2\nmid k}}\epsilon_{k,n}= \frac{\eta_n-\eta_{n-1}}{X_n}
\end{equation*}
and 
\begin{align*}
x_1=\frac{X_n^2(x_2x_3+1)-1}{x_2}=\frac{X_n(\eta_n-\eta_{n-1})-1}{\eta_{n-1}}.
\end{align*}
Hence, we obtain the elements of $V(X_n)_{0,3}(\Z[X_{n-1}])$ for $n\geq 1$.

We can also check the convergence of the PCFs and determine the values by using \cite[Theorem 4.3]{BEJ}.
Indeed, the convergence is checked as same as \cref{appendix}.
Moreover,
for each $\sigma\in\Gal(\mathbb{B}_n/\Q)$, 
we obtain
\[
\left[\overline{\sigma(\frac{(\eta_n-\eta_{n-1})X_n-1}{\eta_{n-1}}),\sigma(\eta_{n-1}),\sigma(\frac{(\eta_n-\eta_{n-1})/X_n-1}{\eta_{n-1}})}\right]
=
\begin{cases}
    \sigma(X_n) & (|\sigma(\eta_n)|>1) \\
    -\sigma(X_n) & (\text{otherwise}).
\end{cases}
\]

\end{proof}


\subsection{An application to the generalized Pell equation}

From \cref{pcfofXn}, we obtain $(1,3)$-type $\mathbb{Z}[X_{n-1}]$-PCF expansions of $X_n$, that is,
\begin{equation}\label{13pcfofXn}
  \left[\frac{(\eta_n-\eta_{n-1})X_n-1}{\eta_{n-1}},\overline{\eta_{n-1},\frac{(\eta_n-\eta_{n-1})/X_n-1}{\eta_{n-1}},\frac{(\eta_n-\eta_{n-1})X_n-1}{\eta_{n-1}}}\right].
\end{equation}
We will show an application to solutions of the generalized Pell equation $x^2-X_n^2y^2=\pm 1$ in $\Z[X_{n-1}]$. 
When $n=1$, the generalized Pell equation coincides with the Pell equation $x^2-2y^2=\pm 1$.
In what follows,
we will consider the solutions of the generalized Pell equation in $\Z[X_{n-1}]$.
Set
\[
W_n\coloneqq\{(u,v)\in\Z[X_{n-1}]^2\mid \text{$u^2-X_n^2v^2=1$ or $u^2-X_n^2v^2=-1$}\}
\]
and 
\[
RE_n=\{\epsilon \in \mathbb{Z}[X_n] \mid N_{n/n-1}(\epsilon)\in \{\pm1\}\}\subset \Z[X_n]^*.   
\]
$RE_n$ is a subgroup of $\Z[X_n]^{*}$ and called the group of relative units.
By Dirichlet's unit theorem,
we obtain the isomorphism as a $\Z$-module
\begin{align*}
RE_n\cong \Z/2\Z\times \Z^{2^{n-1}}.
\end{align*}
In a similar manner to the Pell equation,
there is a bijection
\begin{align*}
W_n\cong RE_n; \ (x,y)\mapsto x+X_ny.
\end{align*}
for each $n\geq 1$.
The set in $W_n$ corresponding to generators of $RE_n$ is called a fundamental solution of the generalized Pell equation.

$RE_n$ is also deeply related to the Weber conjecture.
Let $h_n$ be the class number of $\mathbb{B}_n$ for $n\in\Z_{\geq 1}$.

\begin{conjecture}[{cf.\ \cite[Section 2]{Miller2014}}]\label{conjecture}
For each $n\in\Z_{\geq 1}$,
$h_n=1$. 
\end{conjecture}

Under \cref{conjecture},
we find that the $2$nd convergent of (\ref{13pcfofXn}) gives a fundamental solution of the generalized Pell equation $x^2-X_n^2y^2=\pm 1$ for each $n$.
This is an analogue for \cref{fundamentalsolution}.
Let $(p_2,q_2)$ be the $2$nd convergent of (\ref{13pcfofXn}) and set
\[
A_n=\langle-1, \sigma(p_2+X_nq_2)\mid \sigma\in\Gal(\mathbb{B}_n/\Q)\rangle_{\Z}.    
\]

\begin{proposition}\label{fundsolofXn}
Under \cref{conjecture}, 
we obtain $RE_n=A_n$ for each $n\in\Z_{\geq 1}$.
\end{proposition}

\begin{proof}
Set $X_n=[\overline{a_1,a_2,a_3}]$ corresponding as \cref{pcfofXn}.
Then,
\[
\frac{p_2}{q_2}=\frac{a_1a_2a_3+a_1+a_3}{a_2a_3+1}=\frac{\eta_{n-1}}{(\eta_n-\eta_{n-1})/X_n}    
\]
and we see that $p_2X_n+q_2=\eta_n$. 
Thus,
we obtain
\[
A_n=\langle-1, \sigma(\eta_n)\mid \sigma\in\Gal(\mathbb{B}_n/\Q)\rangle_{\Z}.
\]
In what follows,
we will show $(RE_n : A_n)=h_n/h_{n-1}$ for $n\in\Z_{\geq 1}$ where $(RE_n : A_n)$ is the index.
Set
\[
A_n^{+}=\langle-1, \sigma(\eta_n)\mid \sigma\in\Gal(\mathbb{B}_n/\Q)\rangle_{\Z} \cap\{\epsilon \in \mathbb{Z}[X_n] \mid N_{n/n-1}(\epsilon)=1\}.
\]
By \cite[Lemma 3.2, (2)]{MO2020}, 
we see that $(RE_n : A_n)=(RE_n^+ : A_n^+)$.
Hence, it is sufficient to show $(RE_n^{+} : A_n^{+})=h_n/h_{n-1}$,
which is already shown in \cite[Section 4]{Yoshizaki}.
For convenience,
we will recall the outline of proof.

Let $C_n$ be the group of cyclotomic units in $\mathbb{B}_n$ (cf.\ \cite[Chapter 8]{Washington}).
Then, $A_n^+=RE_n^+\cap C_n$ and the relative norm induces the following exact sequence;
\[
1\rightarrow RE_n^+/A_n^+ \rightarrow \mathbb{Z}[X_n]^*/C_n \rightarrow \mathbb{Z}[X_{n-1}]^*/C_{n-1} \rightarrow 1.
\]
Since $(\Z[X_n]^* : C_n)=h_n$ (cf.\ \cite[Theorem 8.2]{Washington}), 
we obtain $(RE_n^+ : A_n^+)=h_n/h_{n-1}$.
\end{proof}



\appendix
\section{Convergence of PICFs}\label{appendix}

According to \cite[Theorem 4.3]{BEJ}, 
for checking the convergence of $[b_1, \overline{a_1, \dots, a_l}]$,
we should see whether the following conditions hold for $P:=(b_1, a_1,\dots, a_l)$;
\begin{enumerate}
\item\label{condition1} $E(P) \neq \pm I$. 
\item\label{condition2} For all $j=0, \dots,l-1$, $M([a_{j+1},...,a_{j+l}])_{21}\neq 0$ or $\abs{M([a_{j+1},...,a_{j+l}])_{22}} \le 1$.
\item\label{condition3} $(-1)^l\Tr(E(P))^2<0$ or $4\ge (-1)^l\Tr(E(P))^2$.
\end{enumerate}
Here, $a_{j+l}:=a_j$ for $j\geq 1$. 

In the following, we check these conditions for PICFs appeared in \cref{Maintheorem1}. 
\begin{itemize}
\item If $l=1$, we set $P:=(t,2t)$.
(\ref{condition1}) holds since 
\[
E(P)=
\begin{bmatrix}
t & t^2+1 \\ 
1 & t
\end{bmatrix}
\]
and $E(P) \neq \pm I$ for all non-zero integers $t$.

Next, we write down $M([a_{j+1}])_{21}$ for $j=0$.
Then we obtain
\begin{align*}
&M([2t])_{21}=1
\end{align*}
and (\ref{condition2}) holds since $M([a_{j+1}])_{21}$ is non-zero for $t \in \Z\setminus \{0\}$. 

We also find that (\ref{condition3}) holds since
\begin{align*}
(-1)^1\Tr(E(P))^2&=-(2t)^2<0
\end{align*}
for $t\in\Z \setminus \{0\}$.

\item If $l=2$, we set $P:=(st,2s,2st)$ and $P':=(st,s,2st)$. 
\begin{itemize}
\item We will check the conditions for $P$.
(\ref{condition1}) holds since 
\[
E(P)=
\begin{bmatrix}
2s^2t+1 & 2s^3t^2+2st \\ 
2s & 2s^2t+1
\end{bmatrix}
\]
and $E(P) \neq \pm I$ for all non-zero integers $s, t$.

Next, we write down $M([a_{j+1},a_{j+2}])_{21}$ for each $j=0,1$.
Then we obtain
\begin{align*}
&M([2s,2st])_{21}=2st,\\
&M([2st,2s])_{21}=2s,
\end{align*}
and (\ref{condition2}) holds since all $M([a_{j+1},a_{j+2}])_{21}$ are non-zero for $s,t \in \Z\setminus \{0\}$. 

We also obtain 
\begin{align*}
(-1)^2\Tr(E(P))^2&=(4s^2t+2)^2\geq 4
\end{align*}
since $2s^2t+1\neq0$ for $s,t \in \Z \setminus \{0\}$.
Hence (\ref{condition3}) holds for $s,t\in\Z \setminus \{0\}$.

\item We will check the conditions for $P'$. 
(\ref{condition1}) holds since 
\[
E(P')=
\begin{bmatrix}
s^2t+1 & s^3t^2+2st\\ 
s & s^2t+1
\end{bmatrix}
\]
and $E(P') \neq \pm I$ for all non-zero integers $s, t$.

Next, we write down $M([a_{j+1},a_{j+2}])_{21}$ for each $j=0,1$.
Then we obtain
\begin{align*}
&M([s,2st])_{21}=2st,\\
&M([2st,s])_{21}=s,
\end{align*}
and (\ref{condition2}) holds since all $M([a_{j+1},a_{j+2}])_{21}$ are non-zero for $s,t \in \Z\setminus \{0\}$. 

We also obtain 
\begin{align*}
(-1)^2\Tr(E(P'))^2&=(2s^2t+2)^2 \geq 4
\end{align*}
since $(s,t)\neq(\pm1,-1)$ by $m'_2(s,t)\leq 0$. 
Hence (\ref{condition3}) holds.
\end{itemize}

\item If $l=3$, we set $P:=(s+(4s^2+1)t,2s,2s,2(s+(4s^2+1)t))$, 
$Q:=(-2+t,1,-2,-1+2t)$, 
$R:=(-1+t,2,-1,1+2t)$,
$O:=(2+5t,-2,3,3+10t)$ and $L:=(1+5t,3,-2,3+10t)$. 

\begin{itemize}
\item We will check the conditions for $P$.
(\ref{condition1}) holds since 
\begin{align*}
&E(P)=
\begin{bmatrix}
16ts^4 + 4s^3 + 8ts^2 + 3s + t & (4s^2+1)m_3(s,t)\\
4s^2 + 1 & 16ts^4 + 4s^3 + 8ts^3 + 3s + t
\end{bmatrix}
\end{align*}
and $E(P) \neq \pm I$ for all integers $s, t$.

Next, we write down $M([a_{j+1},a_{j+2},a_{j+3}])_{21}$ for each $j=0,1,2$.
Then we obtain
\begin{align*}
&M([2s,2s,2(s+(4s^2+1)t)])_{21}=16ts^3 + 4s^2 + 4ts + 1,\\
&M([2s,2(s+(4s^2+1)t),2s])_{21}=16ts^3 + 4s^2 + 4ts + 1,\\
&M([2(s+(4s^2+1)t),2s,2s])_{21}=4s^2+1.
\end{align*}
and (\ref{condition2}) holds since all $M([a_{j+1},a_{j+2},a_{j+3}])_{21}$ are odd. 

We also obtain 
\begin{align*}
(-1)^3\Tr(E(P))^2&=-(32ts^4 + 8s^3 + 16ts^2 + 6s + 2t)^2\\
&=-4(16ts^4 + 4s^3 + 8ts^2 + 3s + t)^2<0.
\end{align*}
since $16ts^4 + 4s^3 + 8ts^2 + 3s + t\neq 0$.\footnote{
If $16ts^4 + 4s^3 + 8ts^2 + 3s + t = 0$, then $t=-(4s^3)/(4s^2+1)^2 \not\in \Z$ and this is contradiction. 
}
Hence (\ref{condition3}) holds.

\item We will check the conditions for $Q$.

(\ref{condition1}) holds since 
\begin{align*}
    E(Q)=
    \begin{bmatrix}
    -t & -t^2-1 \\ 
    -1 & -t
    \end{bmatrix}
\end{align*}
and $E(Q) \neq \pm I$ for all non-zero integers $t$.

Next, we write down $M([a_{j+1},a_{j+2},a_{j+3}])_{21}$ for each $j=0,1,2$.
Then we obtain
\begin{align*}
&M([1,-2,-1+2t])_{21}=-4t+3,\\
&M([-2,-1+2t,1])_{21}=2t,\\
&M([-1+2t,1,-2])_{21}=-1.
\end{align*}
and  (\ref{condition2}) holds since all $M([a_{j+1},a_{j+2},a_{j+3}])_{21}$ are non-zero.

We also obtain 
\begin{align*}
(-1)^3\Tr(E(P))^2&=-(-2t)^2<0
\end{align*}
for $t\in \Z\setminus \{0\}$ and (\ref{condition3}) holds.

\item  We will check the conditions for $R$.
(\ref{condition1}) holds since 
\begin{align*}
    E(R)=
    \begin{bmatrix}
    -t & -t^2-1 \\
    -1 & -t
    \end{bmatrix}
\end{align*}
and $E(R) \neq \pm I$ for all non-zero integers $t$.

Next, we write down $M([a_{j+1},a_{j+2},a_{j+3}])_{21}$ for each $j=0,1,2$.
Then we obtain
\begin{align*}
&M([2,-1,1+2t])_{21}=-2t,\\
&M([-1,1+2t,2])_{21}=4t+3,\\
&M([1+2t,2,-1])_{21}=-1,
\end{align*}
and (\ref{condition2}) holds since all $M([a_{j+1},a_{j+2},a_{j+3}])_{21}$ are non-zero.

We also obtain 
\begin{align*}
(-1)^3\Tr(E(R))^2&=-(-2t)^2<0
\end{align*}
for $t\in \Z\setminus \{0\}$ and (\ref{condition3}) holds.

\item We will check the conditions for $O$.
(\ref{condition1}) holds since 
\begin{align*}
    E(O)=
    \begin{bmatrix}
    -25t-7 & -125t^2-70t-10\\
    -5 & -25t-7
    \end{bmatrix}
\end{align*}
and $E(O) \neq \pm I$ for all integers $t$.

Next, we write down $M([a_{j+1},a_{j+2},a_{j+3}])_{21}$ for each $j=0,1,2$.
Then we obtain
\begin{align*}
&M([-2,3,3+10t])_{21}=30t+10,\\
&M([3,3+10t,-2])_{21}=-20t-5,\\
&M([3+10t,-2,3])_{21}=-5,
\end{align*}
and (\ref{condition2}) holds since all $M([a_{j+1},a_{j+2},a_{j+3}])_{21}$ are non-zero for all integers $t$.

We also obtain 
\begin{align*}
(-1)^3\Tr(E(O))^2&=-(-50t-14)^2<0
\end{align*}
for all integers $t$ and (\ref{condition3}) holds.
\end{itemize}

\item We will check the conditions for $L$.
(\ref{condition1}) holds since 
\begin{align*}
    E(L)=
    \begin{bmatrix}
    -25t-7 & -125t^2-70t-10\\
    -5 & -25t-7
    \end{bmatrix}
\end{align*}
and $E(L) \neq \pm I$ for all integers $t$.

Next, we write down $M([a_{j+1},a_{j+2},a_{j+3}])_{21}$ for each $j=0,1,2$.
Then we obtain
\begin{align*}
&M([3,-2,3+10t])_{21}=-20t-5,\\
&M([-2,3+10t,3])_{21}=30t+10,\\
&M([3+10t,3,-2])_{21}=-5,
\end{align*}
and (\ref{condition2}) holds since all $M([a_{j+1},a_{j+2},a_{j+3}])_{21}$ are non-zero for all integers $t$.

We also obtain 
\begin{align*}
(-1)^3\Tr(E(L))^2&=-(-50t-14)^2<0
\end{align*}
for all integers $t$ and (\ref{condition3}) holds.

\end{itemize}
\section*{Acknowledgements}
The authors would like to thank Prof.\ Ken-ichi Bannai and Prof.\ Tomokazu Kashio for their careful reading and valuable comments on a draft of this paper.
The authors express their sincere gratitude to Hohto Bekki and Shuji Yamamoto for their valuable comments.
The authors are also grateful to Yoshinosuke Hirakawa and Yuya Matsumoto for their careful reading and many valuable comments and suggestions. 
The first author was supported by JSPS KAKENHI Grant Number JP22J13607 and JST SPRING Grant Number JPMJSP1037.
The second author was supported by JSPS KAKENHI Grant Number JP22J10004.

\end{document}